\documentclass[reqno,a4paper,12pt]{amsart}

\usepackage{mystyle}
\usepackage{etoolbox}

\usepackage[utf8]{inputenc}
\usepackage[T1]{fontenc}
\usepackage{lmodern}
\usepackage[british]{babel}
\usepackage[nopatch=eqnum]{microtype}
\usepackage[headings]{fullpage}
\usepackage{parskip}
\allowdisplaybreaks

\usepackage[pdfencoding=auto,pdfusetitle]{hyperref}
\hypersetup{colorlinks=true, linkcolor=blue!30!black, citecolor=green!30!black, urlcolor=red!45!black}

\usepackage{microtype}

\usepackage{amsfonts,amssymb,bbm,stmaryrd}
\usepackage[scr=boondoxo,scrscaled=1.05]{mathalfa}
\usepackage{amsmath}

\usepackage{mathtools}
\usepackage[noabbrev,capitalise]{cleveref}
\usepackage{centernot,float,subcaption,url,array,booktabs,colortbl}

\usepackage[shortlabels]{enumitem}
\setlist[enumerate,1]{label={\roman*)}}

\usepackage{tikz}
\usetikzlibrary{matrix,positioning}
\usetikzlibrary{calc}
\usetikzlibrary{arrows}
\tikzset{> =stealth}

\newcommand{\addQEDstyle}[2]{\AtBeginEnvironment{#1}{\pushQED{\qed}\renewcommand{\qedsymbol}{#2}}\AtEndEnvironment{#1}{\popQED}}

\theoremstyle{plain}
\newtheorem{theorem}{Theorem}[section]
\newtheorem{lemma}[theorem]{Lemma}
\newtheorem{proposition}[theorem]{Proposition}

\theoremstyle{definition}
\newtheorem{definition}[theorem]{Definition}

\theoremstyle{remark}
\newtheorem{remark}[theorem]{Remark}

\addQEDstyle{example}{$\triangle$}

\renewcommand{\epsilon}{\varepsilon}
\renewcommand{\phi}{\varphi}

\mathchardef\mhyphen="2D

\renewcommand{\land}{\mathrel{\wedge}}
\renewcommand{\lor}{\mathrel{\vee}}

\newcommand{\N}{\mathbb{N}}

\newcommand{\Q}{\mathbb{Q}}
\newcommand{\R}{\mathbb{R}}

\def\U{\mathcal{U}}
\newcommand{\V}{\mathcal{V}}
\newcommand{\B}{\mathcal{B}}

\newcommand{\Sup}{\mathbf{Sup}}
\newcommand{\Frm}{\mathbf{Frm}}
\newcommand{\Loc}{\mathbf{Loc}}

\newcommand{\PreUnifLoc}{\mathbf{PUnifLoc}}

\newcommand{\CUnifLoc}{\mathbf{CUnifLoc}}

\newcommand{\op}{{^\mathrm{\hspace{0.5pt}op}}}

\renewcommand{\O}{\mathcal{O}}

\newcommand{\PL}{\mathrm{P_L}}

\newcommand{\Cvar}{\mathcal{C}}
\newcommand{\st}{\mathrm{st}}
\newcommand{\ModCauchy}{\mathrm{Cauchy}}
\def\slashedrightarrow{\mathrel{\mkern3mu  \vcenter{\hbox{$\scriptscriptstyle+$}}\mkern-12mu{\to}}}

\DeclarePairedDelimiter\abs{\lvert}{\rvert}

\title{Sequences suffice for pointfree uniform completions}

\author[G. Manuell]{Graham Manuell}
\address{Department of Mathematical Sciences, Stellenbosch University, South Africa}
\email{graham@manuell.me}
\thanks{The author acknowledges some financial support from the Centre for Mathematics of the University of Coimbra (UIDB/00324/2020, funded by the Portuguese Government through FCT/MCTES)}

\date{August 2025}


\begin{document}

\begin{abstract}
 Completions of metric spaces are usually constructed using Cauchy sequences. However, this does not work for general uniform spaces, where Cauchy filters or nets must be used instead. The situation in pointfree topology is more straightforward: the correct completion of uniform locales can indeed be obtained as a quotient of a locale of Cauchy sequences.
\end{abstract}

\maketitle
\thispagestyle{empty}

\setcounter{section}{-1}
\section{Introduction}

Uniform spaces provide a general setting in which to discuss uniform continuity and completeness.
In particular, the completion of uniform spaces generalises the previously known completions for metric spaces and topological groups.

However, while metric spaces can be completed via Cauchy sequences, completions of uniform spaces require the use of either nets or filters. For example, the ordinal $\omega_1$ with the order topology is sequentially compact, but not compact. Using the uniformity generated by all finite open covers, we then quickly see that every Cauchy sequence in $\omega_1$ converges, but it is not complete.

Pointfree topology studies topological spaces by means of their lattices of open sets. This allows for the application of algebraic techniques to topology and a very well-behaved constructive theory. Even classically, the pointfree theory is often better behaved than the point-set one. In many ways, the theory of locales parallels that of (quasi-)\allowbreak Polish spaces, but without countability restrictions (see \cite{heckmann2015,chen2020}).

Results in pointfree topology tend to be stable under forcing. Since forcing allows us to make any set countable, we might hope that a pointfree approach to completion would allow us to bypass nets and filters and use sequences for completions of arbitrary uniform locales. On the other hand, every uniform space can be viewed as a uniform locale, so if sequences are not sufficient for spaces, it would be surprising that they somehow work for the more general class of uniform locales. Nonetheless, we will see that this is indeed the case. Thus, we might view the failure of sequences to describe uniform completions to be a pathology of the point-set setting that is rectified by the pointfree approach.

\section{Background}

In this section we will give a brief background to the concepts necessary to understanding this paper. For a more detailed introduction to pointfree topology (and uniform locales in particular) see \cite{PicadoPultr}.
The results in the current paper are constructively valid, but we will not dwell on this in our presentation. An account of the constructive theory of uniform locales can be found in \cite{manuell2024uniform}.

\subsection{Frames and locales}

Frames generalise the open set lattices of topological spaces. A frame is a complete lattice satisfying the distributive law $a \wedge \bigvee_\alpha b_\alpha = \bigvee_\alpha a \wedge b_\alpha$. A \emph{frame homomorphism} is a function between frames which preserves finite meets and arbitrary joins. We denote the category of frames and frame homomorphisms by $\Frm$. A continuous map between topological spaces induces a frame homomorphism \emph{in the reverse direction} between the corresponding frames of open sets by taking preimages. In order to make the morphisms go in the same direction as continuous maps we set $\Loc = \Frm\op$ and call this the category of \emph{locales} and locale morphisms.

We will think of locales as `spaces'. Thus, we will distinguish notationally between a locale $X$ and its corresponding `frame of opens', $\O X$. Of course, formally $X$ and $\O X$ are the same thing, but while an element $u \in \O X$ corresponds to an open of $X$, we want to interpret elements ``$x \in X$'' as being generalised points of the locale in the sense of categorical logic.
If $f\colon X \to Y$ is a locale morphism we write $f^*\colon \O Y \to \O X$ for the corresponding frame homomorphism to evoke the idea of the preimage $f^{-1}$ of a continuous map $f$.

Frames can be presented by generators and relations, just like more conventional algebraic structures. We write $\langle G \mid R\rangle$ for the frame freely generated by the generators $G$ subject to the relation $R$, which consist of formal (in)equalities between elements of the free frame on $G$.
Frame presentations can be understood from a logical perspective as giving an axiomatisation for a geometric theory whose models are the points of the corresponding locale.
Propositional geometric logic has finitary conjunctions and infinitary disjunctions, which are interpreted as finite meets and arbitrary joins in a frame. The generators can be understood as basic propositions, while the relations describe implications between geometric formulae.

As an example let us consider a presentation for the locale of real numbers. Recall that a real number can be described by a \emph{Dedekind cut} --- a pair $(L,U)$ of subsets of $\Q$ satisfying a number of axioms. Here $L$ is understood to be the set of rationals that are smaller than the real in question and $U$ gives the set of rationals larger than the real question. Such a pair can be specified by saying which rationals lie in $L$ and $U$. For each $q \in \Q$ we have generators / basic propositions $[q \in L]$ and $[q \in U]$. These satisfy the following axioms (expressed as \emph{sequents}).
\begin{displaymath}
\begin{array}{r@{\hspace{1.5ex}}c@{\hspace{1.5ex}}l@{\quad}@{}l@{\qquad\quad}r@{}}
  {[q \in L]} &\vdash& {[p \in L]} & \text{ for $p \le q$} & \text{($L$ downward closed)} \\
  {[q \in L]} &\vdash& \bigvee_{p > q} {[p \in L]} & \text{ for $q \in \Q$} & \text{($L$ rounded)} \\
  {\top} &\vdash& \bigvee_{q \in \Q} {[q \in L]} && \text{($L$ inhabited)} \\
  {[p \in U]} &\vdash& {[q \in U]} & \text{ for $p \le q$} & \text{($U$ upward closed)} \\
  {[q \in U]} &\vdash& \bigvee_{p < q} {[p \in U]} & \text{ for $q \in \Q$} & \text{($U$ rounded)} \\
  {\top} &\vdash& \bigvee_{q \in \Q} {[q \in U]} && \text{($U$ inhabited)} \\
  {[p \in L]} \land {[q \in U]} &\vdash& \bot & \text{ for $p \ge q$} & \text{($L$ and $U$ disjoint)} \\
  {\top} &\vdash& {[p \in L]} \lor {[q \in U]} & \text{ for $p < q$} & \text{(locatedness)}
\end{array}
\end{displaymath}
Here the turnstile $\vdash$ can be understood as implication or $\le$, while $\top$ and $\bot$ are true and false (or $1$ and $0$), respectively.
These axioms become the relations in our presentation and the resulting locale agrees with the usual space of real numbers $\R$.
Note that we merely needed to axiomatise the notion of Dedekind cut and we obtained not only the correct points, but also the correct topology.

The product of two locales $X \times Y$ is given by a presentation with generators $\iota_1(u)$ for each $u \in \O X$ and $\iota_2(v)$ for each $v \in \O Y$, and relations enforcing all the relations that already hold in $\O X$ and $\O Y$. We write $u \oplus v = \iota_1(u) \wedge \iota_2(v)$; these are the basic open rectangles of the product.

An embedding of locales $f\colon X \hookrightarrow Y$ is a locale morphism for which $f^*$ is surjective. Equivalence classes of embeddings are called \emph{sublocales}. Sublocales are obtained by adding additional axioms to the geometric theory while leaving the generators alone.
A sublocale is \emph{closed} if it is obtained by adding only axioms of the from $\bigwedge_i g_i \vdash \bot$.

\subsection{Uniform structures}

Uniform locales can be defined by specifying families of \emph{uniform covers}, which consist of opens that we can think of as being `of a similar size'.
A \emph{cover} on a frame $\O X$ is a subset $C \subseteq \O X$ for which $\bigvee C = 1$. A \emph{strong} cover is a cover in which every element is nonzero.
Covers are preordered by refinement: $C_1 \le C_2$ if for every $u \in C_1$ there is a $v \in C_2$ such that $u \le v$.

Given a cover $U$, the \emph{star} of an open $a \in \O X$ is given by \[\st(a,U) = \bigvee \{u \in U \mid a \between u\},\]
where $a \between u$ means that $a \wedge u > 0$. We understand this as making $a$ `a little bigger' (to a degree controlled by the cover).
Sometimes it is useful to write $a\vartriangleleft_U b$ when $\st(a,U) \le b$.
The star of a cover $U$ is defined to be \[U^\star = \{ \st(u, U) \mid u \in C \}.\]

\begin{definition}
 A \emph{pre-uniform locale} is a locale\footnote{Constructively, we additionally assume $X$ is \emph{overt}, as in \cite{manuell2024uniform}. Classically, this holds automatically.} $X$ equipped with is a filter $\U$ of strong covers on $\O X$ (with respect to refinement) such that for every $U \in \U$ there is a $V \in \U$ with $V^\star \le U$.
 Such a filter $\U$ is called a \emph{uniformity}.
 
 A \emph{morphism of pre-uniform locales} $f\colon (X, \U) \to (Y,\V)$ is a morphism of locales $f\colon X \to Y$ such that $f^*[V] \in \U$ for all $V \in \V$. We write $\PreUnifLoc$ for the category of pre-uniform locales.
\end{definition}

An important class of pre-uniform locales is those arising from a metric $d\colon X \times X \to \R_{\ge 0}$. For each $\epsilon > 0$, there is a uniform cover consisting of the $u \in \O X$ of diameter less than $\epsilon$ --- i.e.\ those such that $u \oplus u \le d^*([0,\epsilon))$. The filter generated by these covers is the \emph{metric uniformity} on $X$.

We note that a \emph{uniform space} is simply a pre-uniform structure on a discrete locale. Such a uniformity also induces a topology coarser than the discrete topology on the underlying set of the uniform space. Something similar happens in the pointfree setting more generally.

If $(X,\U)$ is a pre-uniform locale, we define the uniformly below relation on $\O X$ by
\[a\vartriangleleft b \iff \exists U \in \U.\ a\vartriangleleft_U b \iff \exists U \in \U.\ \st(a,U) \le b.\]
This relation is compatible with the order, stable under finite meets and joins, transitive and interpolative (for the last condition, we in particular have $a \vartriangleleft_U b \implies a \vartriangleleft_V \st(a,V) \vartriangleleft_V b$ when $V^\star \le U$).
We obtain a subframe of $\O X$ consisting of the opens $u$ such that $u = \bigvee_{v \vartriangleleft u} v$.
We sometimes want this induced subframe to agree with the intrinsic finer topology of the locale $X$. This leads us to define \emph{uniform} locales.
\begin{definition}
 A \emph{uniform locale} is a pre-uniform locale $(X, \U)$ such that every $u \in \O X$ satisfies $u = \bigvee_{v \vartriangleleft u} v$.
\end{definition}

We say a morphism of (pre-)uniform locales $f\colon (X, \U) \to (Y,\V)$ is a \emph{uniform embedding} if $f$ is a locale embedding and the covers $f^*[V]$ for $V \in \V$ form a filter base for $\U$.

\subsection{Completeness}

A uniform locale $X$ is said to be \emph{complete} if every uniform embedding of $X$ into another uniform locale is closed. The category $\CUnifLoc$ of complete uniform locales is a reflective subcategory of the category $\PreUnifLoc$ of pre-uniform locales. The reflector is called the \emph{completion} functor $\Cvar$.

\begin{theorem}\label{prop:completion_via_filters}
 Let $(X,\U)$ be a pre-uniform locale. The completion of $X$ has an underlying locale with a presentation given by a generator $[a \in F]$ for each $a \in \O X$ and the following relations:
 \begin{enumerate}
 \item $[1 \in F] = 1$,
 \item $[a \wedge b \in F] = [a \in F] \wedge [b \in F]$,
 \item $[a \in F] \le \bigvee\{1 \mid a > 0\}$\footnote{This join is indexed by a subsingleton. Classically, it is equal to $1$ if $a > 0$ and $0$ if $a = 0$. Here only the $a = 0$ case gives a nontrivial condition, namely $[0 \in F] = 0$.},
 \item $\bigvee_{u \in U} [u \in F] = 1$ for each $U \in \U$,
 \item $[a \in F] \le \bigvee_{b \vartriangleleft a} [b \in F]$.
 \end{enumerate}
 The uniformity on $\Cvar X$ is generated by the covers of the form $\{[u \in F] \mid u \in U\}$ for $U \in \U$
 and the unit of the adjunction $\gamma_X\colon X \to \Cvar X$ is defined by $\gamma_X^*\colon [a \in F] \mapsto \bigvee_{b \vartriangleleft a} b$.
\end{theorem}

\begin{remark}
 The presentation above describes the geometric theory of regular Cauchy filters on $X$. The generator $[a \in F]$ says that the filter (here called $F$) contains $a$. The first two axioms say $F$ is a filter, the third that it is proper, the forth that it is Cauchy and the final one that it is regular. The regularity axiom is to avoid needing to take a quotient later (see \cite{manuell2024uniform} for details). Regular Cauchy filters are also how completions of uniform locales are usually constructed in the spatial setting, but note that by using presentations we get the correct topology on the completion `for free'.
\end{remark}

The map $\gamma_X$ is dense, and if $(X, \U)$ is uniform, then $\gamma$ is a uniform embedding. (In general, $\gamma$ can fail to be a locale embedding, but it is still \emph{initial} with respect to the forgetful functor $\PreUnifLoc \to \Loc$, which is the other half of the definition on a uniform embedding.)

\subsection{Quotients and suplattices}\label{sec:quotients}

The usual construction of the completion of a metric space via Cauchy sequences involves quotienting a set of Cauchy sequences to identify sequences that will converge to the same limit. In a similar way, our construction will need to take a quotient of a \emph{locale} of Cauchy sequences. See Vickers' construction of the localic completion of a metric space (not a metric locale) \cite[\S 7]{vickers1997localic} for a similar approach in a special case.

General quotients of locales can be badly behaved, but luckily the type of quotient we will need is of a special kind. Let us first look at an even more specific kind of quotient.

\begin{definition}
 A locale morphism $f\colon X \to Y$ is called an \emph{open quotient} if the frame map $f^*\colon \O Y \to \O X$ has a left adjoint retraction $f_!\colon \O X \to \O Y$ satisfying the Frobenius condition $f_!(a \wedge f^*(b)) = f_!(a) \wedge b$.
\end{definition}

Since $f_!$ is a left adjoint, it preserves arbitrary joins. Such a map between complete lattices is called a \emph{suplattice homomorphism}. The category of complete lattices and suplattice homomorphisms is called $\Sup$. The type of quotient we will need replaces the left adjoint $f_!$ in the definition of open quotient with a more general suplattice homomorphism.

\begin{definition}
 We will call a locale morphism $f\colon X \to Y$ a \emph{lower triquotient} if there is a suplattice homomorphism $f_\#\colon \O X \to \O Y$ such that $f_\#(a \wedge f^*(b)) = f_\#(a) \wedge b$ and $f_\#(1) = 1$. We call $f_\#$ a \emph{triquotiency assignment}.
\end{definition}
Note that in the presence of the other conditions, $f_\#(1) = 1$ is equivalent to $f_\#$ being a left inverse of $f^*$ in $\Sup$. As the name implies, these lower triquotient maps are a special case of the triquotient maps defined by Plewe in \cite{plewe1997localic}. There it is shown that triquotient maps are pullback-stable regular epimorphisms and even effective descent morphisms.

It is possible to get more `topological' intuition for these quotients. The \emph{lower powerlocale} $\PL X$ of a locale $X$ is the pointfree analogue of the lower Vietoris hyperspace. It is given by the composite of the forgetful functor from $\Frm$ to $\Sup$ and its left adjoint. This adjunction allows us to understand suplattice maps from $\O X \to \O Y$ as locale maps from $Y$ to $\PL X$.
The points of the lower powerlocale are closed sublocales of $X$ (but see \cite{bunge1996lower} for the constructive situation) and the subbasic opens are of the form $\lozenge a$ indicating whether the sublocale intersects with the open $a$. So maps $Y \to \PL X$ can be understood to be multivalued maps from $Y$ to $X$. (Compare the lower hemicontinuous set-valued functions in classical analysis.)

The fact that $f_\#$ is a left inverse of $f^*$ means that, up to taking closures, it corresponds to a \emph{multivalued section} of $f$. By ``up to taking closures'' we mean that each $y \in Y$ maps to a sublocale that is contained in the preimage of the closure of $y$, but might not be contained in the fibre of $y$ itself.
However, as discussed in \cite{vickers1994pointless}, the Frobenius condition allows us to remove these ``up to closure'' caveats.
Thus, a lower triquotient is precisely a locale map with a multivalued section in the above sense. Incidentally, an open map is one where the multivalued section can be chosen to pick out precisely the fibre of each point.

The locales in our construction of the completion will given by frame presentations. Thus, to find the triquotiency assignment we need a way to define suplattice homomorphisms in terms of frame presentations.

First note that suplattices themselves can be presented by generators and relations \cite{joyal1984galois}.
Next observe the set of generators in a presentation might come equipped a (pre-)order structure, in which case we usually ask for this order to be preserved by the inclusion of the generators into the resulting algebraic structure by imposing additional implicit relations.
We then have the following theorem.

\begin{theorem}[Coverage theorem \cite{vickers2006compactness}]\label{thm:sup_coverage}
 Consider a frame presentation \[\langle G \text{ preordered} \mid R_1 \sqcup R_2\rangle_\Frm\] where $G$ is a preordered set and the (non-implicit) relations are divided into two sets $R_1$ and $R_2$. Suppose $R_1$ consists of the relations $1 \le \bigvee G$ and $a \wedge b \le \bigvee\{c \in G \mid c \le a,b\}$ for $a,b \in G$,
 whereas $R_2$ consists of relations of the form $a \le \bigvee A$. Furthermore, suppose that whenever $a \le \bigvee A$ is a relation in $R_2$ and $b \le a$ , there is relation $b \le \bigvee B$ in $R_2$ for some refinement $B$ of $A$ such that $B \subseteq {\downarrow} b$.
 Then there is an order isomorphism \[\langle G \text{ preordered} \mid R_1 \sqcup R_2\rangle_\Frm \cong \langle G \text{ preordered} \mid R_2\rangle_\Sup.\]
\end{theorem}
This will allow us to easily define a suplattice homomorphism out of such a presented frame by specifying it on generators.

\section{The locale of Cauchy sequences}\label{sec:locale_of_cauchy_sequences}

Let $(X, \U)$ be a pre-uniform locale and let $\B$ be a base for the uniformity. We wish to construct a locale of a Cauchy sequences in $X$. We can start by considering the locale of \emph{all} sequences $X^\N$. This is simply the countably infinite product of $X$ with itself. For each open $u \in \O X$ there is a subbasic open of $X^\N$ which we might call $[s(n) \in u]$ that contains the sequences whose $n$th term lies in $U$.

To cut out a sublocale of Cauchy sequence we need to find a family of geometric sequents that express when a sequence is Cauchy.
A sequence $s\colon \N \to X$ is \emph{Cauchy} if \[\forall U \in \B.\ \exists N \in \N.\ \forall n,n' \ge N.\ \exists u \in U.\ s(n) \in u \land s(n') \in u.\]
That is, if for every uniform cover, there is a point in the sequence after which any two elements lie inside the same open from the cover. This expresses the idea that the terms in the sequence eventually get arbitrarily close together with respect to the uniformity.

We can get rid of one quantifier using $s(n) \in u \land s(n') \in u \iff (s(n), s(n')) \in u \oplus u$ and $\exists u \in U.\ (s(n),s(n')) \in u \oplus u \iff (s(n),s(n')) \in \bigvee_{u \in U} u \oplus u$.
So the condition becomes \[\forall U \in \B.\ \exists N \in \N.\ \forall n,n' \ge N.\ (s(n),s(n')) \in \bigvee_{u \in U} u \oplus u.\]
However, this is still apparently too logically complex to be expressed in geometric logic.

Vickers sidesteps this issue in his construction from \cite{vickers1997localic} of the localic completion of a metric space by using \emph{rapidly converging} Cauchy sequences.
Rather than asking for the mere existence of some $N$ for each measure of closeness, for a rapidly converging Cauchy sequence we require the elements to get closer together a some fixed (exponential) rate:
$\forall N \in \N.\ \forall n, n' \ge N.\ d(s(n), s(n')) < 2^{-N}$. This can be expressed geometrically: we add one axiom for each $N, n, n' \in \N$ such that $n,n' \ge N$.
We cannot use this trick in our situation, however, since we lack a metric and there is no way to make sense of an exponential rate of convergence for sequences in a general (pre-)uniform locale.

Instead we \emph{Skolemise} the definition of Cauchy sequence.
This means we replace the complex definition above with one only involving universal quantifiers.
The idea is to move quantifiers past each other by replacing subformulas of the form $\forall x.\ \exists y.\ \phi(x,y)$ with $\exists f.\ \forall x.\ \phi(x,f(x))$.
Eventually we obtain an expression with all the existential quantifiers on the outside. We then take these existentially quantified variables as additional \emph{data}.

Applying this to our definition of Cauchy sequence $s\colon \N \to X$ we get
\[\exists m \in \N^\B.\ \forall U \in \B.\ \forall n,n' \ge m(U).\ (s(n),s(n')) \in \bigvee_{u \in U} u \oplus u.\]
A function $m\colon \N \to \B$ satisfying this is called a \emph{modulus of convergence} or \emph{modulus of Cauchyness} for the sequence $s$. Such a modulus is an explicit measure of how quickly the terms of the sequence approach each other. Incorporating the modulus into the data we obtain a definition of a \emph{modulated Cauchy sequence} as a pair $(s,m)$ such that
\[\forall U \in \B.\ \forall n,n' \ge m(U).\ (s(n),s(n')) \in \bigvee_{u \in U} u \oplus u.\]
This is finally in a form that we will be able to formulate as a geometric theory. (Since we were already planning on quotienting the locale of Cauchy sequences, adding additional data that will just be quotiented away anyway is not a problem.)

There is one final wrinkle. Skolemisation makes use of the axiom of choice. Indeed, the modulus $m$ is precisely a choice function. In the pointfree setting, choice is usually best avoided, but it is not difficult to modify the above approach: simply take $m$ to be a left-total relation (i.e.\ a multivalued function) instead of a function. This means we do not choose just one $n$ for each $U$, but multiple. Otherwise the definition is unchanged.\footnote{
 Actually the above paragraph is misleading, since there is a pointfree version of the axiom of choice which is constructively valid (see \cite[Proposition 2.3.7]{henry2016corrected}).
 Henry's choice theorem does require that the index set $\B$ have decidable equality, which we would like to avoid.
 Nonetheless, I believe it is likely possible to make things work in general.
 The real reason we do not use functions for moduli is for simplicity. It ensures that the space of moduli is overt and leads to technical simplifications later on.}

We may now find a presentation for the locale of modulated Cauchy sequences. We first construct the locale of pairs $(s,m)$ where $s$ is a sequence $s\colon \N \to X$ and $m$ is a left-total relation $m \colon \B \slashedrightarrow \N$. We have generators $[s(n) \in u]$ for each $n \in \N$ and $u \in \O X$ and generators $[m(U) = k]$ for each $U \in \B$ and $k \in \N$. These satisfy
\[\bigvee_\alpha \bigwedge_{u \in F_\alpha} [s(n) \in u] = [s(n) \in \bigvee_\alpha \bigwedge F_\alpha]\]
for each $n \in \N$ and each family $(F_\alpha)_\alpha$ of finite subsets of $\O X$ to ensure that each factor of $X^\N$ has the topology of $X$, and
\[1 \le \bigvee_{k \in \N} {[m(U) = k]}\]
for each for $U \in \B$ so that $m$ is left total (i.e.\ for each $U \in \B$ there is a $k \in \N$ that it maps to / that is related to it).

Now we cut this locale down to the locale of modulated Cauchy sequences by imposing the Cauchyness axiom:
\[[m(U) = k] \le \bigvee_{u \in U} [s(n) \in u] \wedge [s(n') \in u]\]
for $U \in \B$, $k \in \N$ and $n,n' \ge k$. This is simply the translation of the Skolemised condition we found above into geometric logic.
In summary we have the following definition.
\begin{definition}\label{def:modcauchy}
 Let $X$ be a pre-uniform locale with base $\B$ for the uniformity. The locale $\ModCauchy(X)$ of modulated Cauchy sequences is given by a presentation
 with generators $[s(n) \in u]$ for $n \in \N$ and $u \in \O X$ and $[m(U) = k]$ for $U \in \B$ and $k \in \N$, and the following relations:
\begin{enumerate}
 \item $\bigvee_\alpha \bigwedge_{u \in F_\alpha} [s(n) \in u] = [s(n) \in \bigvee_\alpha \bigwedge F_\alpha]$ for $n \in \N$ and formal expression $\bigvee_\alpha \bigwedge F_\alpha$,
 \item $1 \le \bigvee_{k \in \N} {[m(U) = k]}$ for $U \in \B$,
 \item $[m(U) = k] \le \bigvee_{u \in U} [s(n) \in u] \wedge [s(n') \in u]$ for $U \in \B$, $k \in \N$ and $n,n' \ge k$.
\end{enumerate}
\end{definition}
\begin{remark}
 This definition \emph{does} depend on the base $\B$ we choose for the uniformity, since this affects the domain of modulus. However, these differences will disappear once we take the quotient. If a canonical choice is desired, $\B$ can always be taken to be equal to the entire uniformity $\U$, though other choices are usually more convenient.

 If $X$ is given by a presentation and the basic uniform covers only involve those generators, we may restrict the generators $[s(n) \in u]$ to those with generators $u$ and the formal expressions in (i) to basic relations without changing the resulting locale. I believe $\ModCauchy(X)$ can then be seen to be stable under change of base topos.
\end{remark}

\section{The limit map}

We will now construct a map $q\colon \ModCauchy(X) \to \Cvar X$ that `takes the limit' of the Cauchy sequences.
How might we do this? Well, we previously described $\Cvar X$ in terms of regular Cauchy filters, so we must associate a regular Cauchy filter to each modulated Cauchy sequence.

There is a standard way to turn a sequence into a filter: the \emph{filter of tails}. The filter of tails of a sequence $s\colon \N \to X$ is $\{u \in \O X \mid \exists N \in \N.\ \forall n \ge N.\ s(n) \in u\}$. Even if $s$ is convergent, this filter might fail to be regular, but there is also a standard way to regularise a Cauchy filter $F$ --- simply consider $\{u \in \O X \mid \exists u' \in F.\ u' \vartriangleleft u \}$.
Putting these together we have \[(s,m) \mapsto \{u \in \O X \mid \exists u' \in \O X.\ u' \vartriangleleft u \land \exists N \in \N.\ \forall n \ge N.\ s(n) \in u'\}.\]
We will show that this sends (modulated) Cauchy sequences to regular Cauchy filters, and furthermore, that this assignment is actually geometric (and hence continuous).
Note that
\begin{align*}
 q((s,m)) \in [u \in F] &\iff u \in q((s,m)) \\
                        &\iff \exists u' \in \O X.\ u' \vartriangleleft u \land \exists N \in \N.\ \forall n \ge N.\ s(n) \in u'.
\end{align*}
At this point we reach a snag: the universal quantification over $n \ge N$ is not geometric. Luckily, here we can make use of Cauchyness. If we need all elements of the sequence to eventually lie in $u'$, then it is sufficient to find one element $s(k)$ that is well inside $u'$ at a point in the sequence where all the subsequent terms are very close to $s(k)$. Specifically, we ask for some $k \in \N$, some basic uniform cover $U \in \B$ and some $v \vartriangleleft_U u'$ such that $s(k) \in v$ and $k \ge m(U)$.
(Then for any $n \ge k$ we have a $w \in U$ with $s(k) \in w \wedge v$ and $s(n) \in w$, so that $s(n) \in w \le \st(v,U) \le u'$.)
In fact, this condition is also necessary, so long as we can vary $u'$. (Suppose $s(n) \in u'$ for every $n \ge N$ and that $u' \vartriangleleft_V u$. By interpolating we have $u' \vartriangleleft_U u'' \vartriangleleft u$ for some $U \in \B$. Now consider $k' \in \N$ such that $m(U) = k'$ and set $k = \max(k',N)$. Note that $s(k) \in u'$ since $k \ge N$ and so we satisfy the required condition with $v = u'$ and with $u''$ in place of of $u'$.)

This is now a condition we can phrase geometrically:
\begin{align*}
 q((s,m)) \in [u \in F] &\iff \exists U \in \B.\ \exists v \vartriangleleft_U u' \vartriangleleft u.\ \exists k \in \N.\ k \ge m(U) \land s(k) \in v \\
                        &\iff (s,m) \in \bigvee_{U \in \B \vphantom{k'}} \bigvee_{\; v \vartriangleleft_U u' \vartriangleleft u \vphantom{k'}} \bigvee_{\; k' \le k \in \N} [m(U) = k'] \wedge [s(k) \in v].
\end{align*}
In other words we have arrived at the following definition.
\begin{definition}\label{def:limit_map}
 Let $X$ be a pre-uniform locale with base $\B$ for the uniformity. The limit map $q\colon \ModCauchy(X) \to \Cvar X$ is defined by
 \[q^*([u \in F]) = \bigvee_{U \in \B \vphantom{k'}} \bigvee_{\; v \vartriangleleft_U u' \vartriangleleft u \vphantom{k'}} \bigvee_{\; k' \le k \in \N} [m(U) = k'] \wedge [s(k) \in v].\]
\end{definition}

We do still need to check that this gives a well-defined frame homomorphism. Let us do so now.

\begin{lemma}
 There is a unique frame homomorphism $q^*\colon \O\Cvar X \to \O \ModCauchy(X)$ satisfying the equation in \cref{def:limit_map}.
\end{lemma}
\begin{proof}
 We must show that $q^*$ preserves the relations in the presentation of $\O\Cvar X$ from \cref{prop:completion_via_filters}.
 
 (i) Expanding the definition of $q^*$, considering any $U \in \B$, taking $v = u' = u = 1$ and $k' = k$, and noting that $[s(k) = 1] = 1$, we have $q^*([1 \in F]) \ge \bigvee_{k \in \N} [m(U) = k] \ge 1$, where the second inequality is simply the left-totality condition of \cref{def:modcauchy}. Thus, relation (i) is preserved.
 
 (ii) It is clear that $q^*([u \in F])$ is monotonic in $u$. We show $q^*([u_1 \in F]) \wedge q^*([u_2 \in F])) \le q^*([u_1 \wedge u_2 \in F])$.
 Suppose $U_i, V_i \in \B$, $v_i \vartriangleleft_{U_i} u'_i \vartriangleleft_{V_i} u_i$ and $k'_i \le k_i \in \N$ for $i=1,2$. We want $[m(U_1) = k'_1] \wedge [m(U_2) = k'_2] \wedge [s(k_1) \in v_1] \wedge [s(k_2) \in v_2] \le q^*([u_1 \wedge u_2 \in F])$. The idea is now that we want to replace $k_1$ and $k_2$ with the same $n$, larger than both of them.
 
 Take $W \in \B$ such that $W^\star \le V_1 \wedge V_2$.
 By the Cauchyness axiom of \cref{def:modcauchy}, we have $[m(U_i) = k'_i] \le \bigvee_{w_i \in U_i} [s(n) \in w_i] \wedge [s(k_i) \in w_i]$ for any $n \ge \max(k'_1,k'_2)$.
 So $\bigwedge_{i\in\{1,2\}} [m(U_i) = k'_i] \le \bigvee_{w_1 \in U_1} \bigvee_{w_2 \in U_2} [s(n) \in w_1 \wedge w_2] \wedge \bigwedge_{i\in \{1,2\}} [s(k_i) \in w_i]$ for such an $n$.
 Now by left totality, $\bigvee_{n' \in \N} [m(W) = n'] = 1$ and hence $\bigwedge_{i\in\{1,2\}} [m(U_i) = k'_i] \le \bigvee_{w_1 \in U_1} \bigvee_{w_2 \in U_2} \bigvee_{n' \in \N} [m(W) = n'] \wedge [s(\max(n',k_1,k_2)) \in w_1 \wedge w_2] \wedge \bigwedge_{i\in\{1,2\}} [s(k_i) \in w_i]$.
 
 Now note $[s(k_i) \in w_i] \wedge [s(k_i) \in v_i] = [s(k_i) \in w_i\wedge v_i] \le \bigvee\{1 \mid w_i \between v_i\}$, where this last inequality follows from applying frame map $a \mapsto [s(k_i) \in a]$ to $w_i \wedge v_i \le \bigvee\{1 \mid w_i \wedge v_i > 0\}$ in $\O X$.
 But $w_i \between v_i$, together with $w_i \in U_i$ and $v_i \vartriangleleft_{U_i} u'_i$, implies $w_i \le u'_i \vartriangleleft_{V_i} u_i$, and hence $w_1 \wedge w_2 \vartriangleleft_W u'' \vartriangleleft_W u_1 \wedge u_2$ for some $u''$ by using $\wedge$-stability and then interpolation.
 Thus, we have
 \begin{align*}
  &\bigwedge_{i\in\{1,2\}} [m(U_i) = k'_i] \wedge [s(k_i) \in v_i] \\
  &\quad\le \!\bigvee_{\substack{w_1 \in U_1 \\ w_2 \in U_2}}\, \bigvee_{n' \in \N} [m(W) = n'] \wedge [s(\max(n',k_1,k_2)) \in w_1\wedge w_2] \wedge \bigvee\{1 \mid w_1 \between v_1 \text{ and } w_2 \between v_2\} \\
  &\quad\le \bigvee_{w \vartriangleleft_W u'' \vartriangleleft u_1 \wedge u_2} \, \bigvee_{n' \in \N} [m(W) = n'] \wedge [s(\max(n',k_1,k_2)) \in w] \\
  &\quad\le \bigvee_{w \vartriangleleft_W u'' \vartriangleleft u_1 \wedge u_2} \, \bigvee_{n' \le n \in \N} [m(W) = n'] \wedge [s(n) \in w] \\
  &\quad\le q^*([u_1 \wedge u_2 \in F]),
 \end{align*}
 as required.
 
 (iii) Take $U \in \B$, $v \vartriangleleft_U u' \vartriangleleft u$ and $k' \le k \in \N$. We must show $[m(U) = k'] \wedge [s(k) \in v] \le \bigvee\{1 \mid u > 0\}$.
 But $[m(U) = k'] \wedge [s(k) \in v] \le [s(k) \in v] \le \bigvee\{1 \mid v > 0\} \le \bigvee\{1 \mid u > 0\}$ and so we are done.
 
 (iv) To show Cauchyness, we take $U \in \B$ and prove $\bigvee_{u \in U} q^*([u \in F]) = 1$. Take $V, U' \in \B$ such that $V^\star \le U'$ and $U'^\star \le U$.
 The intuition is to consider $v \in V$ such that $s(m(V)) \in v$.
 More formally, we note $\bigvee_{k \in \N} [m(V) = k] = 1$ and that $\bigvee_{v \in V} [s(k) \in v] = 1$ for each $k \in \N$ (since $\bigvee V = 1$). Hence, $\bigvee_{v \in V} \bigvee_{k \in \N} [m(V) = k] \wedge [s(k) \in v] = 1$. Now note that $v \in V$ implies $v \le \st(v,V) \in V^\star$ and so $v \vartriangleleft_V u'$ for some $u' \in U'$ and similarly $u' \vartriangleleft u$ for some $u \in U$. Therefore,
 \begin{align*}
  1 &= \bigvee_{v \in V} \bigvee_{k \in \N} [m(V) = k] \wedge [s(k) \in v] \\
    &\le \bigvee_{u \in U} \bigvee_{v \vartriangleleft_V u' \vartriangleleft u} \bigvee\nolimits_{k \in \N} [m(V) = k] \wedge [s(k) \in v] \\
    &\le \bigvee_{u \in U} q^*([u \in F]),
 \end{align*}
 as required.
 
 (v) The regularity condition essentially holds construction and by interpolating $u' \vartriangleleft u$ to $u' \vartriangleleft u'' \vartriangleleft u$.
\end{proof}

\section{The completion as a quotient}

We will now show that the limit map $q\colon \ModCauchy(X) \to \Cvar X$ is a well-behaved quotient map --- specifically, a lower triquotient.
This will allow us to conceptualise the completion $\Cvar X$ as a quotient of $\ModCauchy(X)$ by the kernel equivalence relation of $q$.

We show $q$ is a lower triquotient by describing a triquotiency assignment, which we view as defining a multivalued section as described in \cref{sec:quotients}.
We wish to give, for each point of the completion, a nontrivial geometrically-definable collection of modulated Cauchy sequences which converge to it.

In order for this collection to admit a geometric definition we will choose it to contain the modulated Cauchy sequences that converge to the point particularly quickly (as measured by the modulus).
For a point of $\Cvar X$ given by a regular Cauchy filter $F$, we might consider the modulated sequences $(m,s)$ such that
$\forall U \in \B.\ \forall n \ge m(U).\ \exists u \in U.\ u \in F \land s(n) \in u$. However, we actually want the sequences to converge even faster than this --- namely, we ask
$\forall U \in \B.\ \exists V \in \B.\ V^\star \le U \land \forall n \ge m(U).\ \exists v \in V.\ v \in F \land s(n) \in v$.

The reason behind this is due to a subtlety in the above `modulus of convergence' as compared to a modulus of Cauchyness. A sequence $s$ in a metric space converges to $x$ if for all $\epsilon > 0$ there is an $N \ge \N$ such that $\forall n \ge N.\ d(s(n),x) < \epsilon$. The corresponding modulus $m$ sends each $\epsilon$ to an appropriate such $N$. Every convergent sequence is Cauchy: for $\epsilon > 0$ and $n,n' \ge m(\epsilon)$ we have $d(s(n),s(n')) \le d(s(n),x) + d(x,s(n')) < 2\epsilon$. But since we had to multiply $\epsilon$ by 2, $m$ is not necessarily a modulus of Cauchyness for $s$.
To ensure that $(m,s)$ is indeed a modulated Cauchy sequence, we must use replace $\epsilon$ with $\epsilon/2$ in the definition of convergence. In the uniform setting, replacing $\epsilon$ with $\epsilon/2$ corresponds to replacing a uniform cover $U$ with a uniform cover $V$ such that $V^\star \le U$.

Intuitively we are using the assignment
\begin{align*}
 \overline{q_\#}\colon F \mapsto \{(s,m) \in \ModCauchy(X) \mid {} & \forall U \in \B.\ \exists V \in \B.\ V^\star \le U \land \\
 & \forall n \in \N.\ m(U) \le n \implies \exists v \in V.\ v \in F \land s(n) \in v \}.
\end{align*} 
What does this do to opens? The generating opens of $\PL \ModCauchy(X)$ are of the form $\lozenge ([s(\vec{n}) \in \vec{u}] \wedge [m(\vec{U}) = \vec{k}])$ with $S \in \lozenge ([s(\vec{n}) \in \vec{u}] \wedge [m(\vec{U}) = \vec{k}]))$ indicating that $S \between (\bigwedge_i [s(n_i) \in u_i] \wedge \bigwedge_j [m(U_j) = k_j])$.
For simplicity assume the $n_i$'s are distinct and each $u_i > 0$.
We wish to know for which $F \in \Cvar X$ we have $\overline{q_\#}(F) \between [s(\vec{n}) \in \vec{u}] \wedge [m(\vec{U}) = \vec{k}]$. Consider a modulated Cauchy sequence $(s,m)$ lying in the intersection. The open specifies the sequence $s$ and the modulus $m$ at finitely many places.
How does asking $(s,m) \in \overline{q_\#}(F)$ interact with these constraints?

We have $\forall U \in \B.\ \exists V \in \B.\ V^\star \le U \land
 \forall n \in \N.\ m(U) \le n \implies \exists v \in V.\ v \in F \land s(n) \in v$.
In particular, this holds for each $U_j \in \vec{U}$ (which then gives a $V_j \in \B$ with $V_j^\star \le U_j$) and each $n_i \in \vec{n}$ such that $k_j \le n_i$.
This means we have a $v_i^j \in V_j$ such that $v_i^j \in F$ and $s(n_i) \in v_i^j$.
For fixed $i$, this entails $s(n_i) \in v_i \coloneqq \bigwedge_{j \colon k_j \le n_i} v^j_i$, which in turn lies in $\bigcap_{j \colon k_j \le n_i} {\downarrow} V_j$ and in $F$ (since $F$ is a filter).
Now note that to have $s(n_i) \in v_i$ and $s(n_i) \in u_i$ we require $v_i \between u_i$, which suggests
\[\overline{q_\#}^{\,*}(\lozenge ([s(\vec{n}) \in \vec{u}] \wedge [m(\vec{U}) = \vec{k}])) = \bigvee\nolimits_{\vec{V} \in \B^{\abs{\vec{k}}} \colon V_j^\star \le U_j} \, \bigwedge\nolimits_i \bigvee \{[v_i \in F] \mid v_i \in \bigcap_{\mathclap{j \colon k_j \le n_i}} {\downarrow} V_j ,\, v_i \between u_i\}. \]
Distributing the meet past the join and taking $q_\#(u) = \overline{q_\#}^{\,*}(\lozenge u)$, we arrive at the following definition.
\begin{definition} \label{def:triquotiency}
 Let $X$ be a pre-uniform locale with base $\B$ for the uniformity. The suplattice map $q_\#\colon \O \ModCauchy(X) \to \O \Cvar X$ is defined by
 \[q_\#([s(\vec{n}) \in \vec{u}] \wedge [m(\vec{U}) = \vec{k}]) = \bigvee_{\substack{\vec{V} \in \B^{\abs{\vec{k}}} \\ V_j^\star \le U_j}}\, \bigvee_{\substack{\vec{v} \in (\cup \B)^{\abs{\vec{n}}} \\ v_i \in {\downarrow} V_j \text{ for } k_j \le n_i \\ v_i \between u_i}} \!\! \bigwedge\nolimits_i [v_i \in F]\]
 where $\vec{n}$ is a finite list of district natural numbers and $\vec{u}$ is a corresponding list of nonzero opens.
\end{definition}
\begin{lemma}
 There is indeed a unique suplattice homomorphism $q_\#\colon \O \ModCauchy(X) \to \O \Cvar X$ satisfying the equation in \Cref{def:triquotiency}.
\end{lemma}
\begin{proof}
 We want to make use of the coverage theorem (\cref{thm:sup_coverage}). The presentation of $\O\ModCauchy(X)$ in \cref{def:modcauchy} is not of the correct form, so we start by replacing it with an equivalent presentation that is.
 
 Our original presentation can thought of as describing $\ModCauchy(X)$ as a sublocale of the product of a locale $X^\N$ of sequences and a locale of left-total relations from $\B$ to $\N$.
 Let us first focus on the case of $X^\N$, which involves the generators $[s(n) \in u]$. Instead of using generators $[s(n) \in u]$ for all $u \in \O X$ we use generators $[s(\vec{n}) \in \vec{u}]$ where $\vec{n}$ denotes a finite list of distinct natural numbers, and $\vec{u}$ is a list of opens $u_i > 0$ of the same length. In $\O X^\N$ we will have $[s(\vec{n}) \in \vec{u}] = \bigwedge_i [s(n_i) \in u_i]$ where $i$ ranges over the list indices. We preorder these generators by \[ [s(\vec{n}) \in \vec{u}] \le [s(\vec{n}') \in \vec{u}'] \iff \forall i'.\ u'_{i'} = 1 \lor \exists i.\ n_i = n'_{i'} \land u_i \le u'_{i'}. \]
 As for the relations, we take the necessary ``$R_1$'' relations from \cref{thm:sup_coverage} in addition to ``$R_2$'' relations of the form \[ [s(\vec{n}) \in u_1,\dots, \bigvee\nolimits_\alpha u_\iota^\alpha,\dots,u_\ell] \le \bigvee\nolimits_\alpha [s(\vec{n}) \in u_1,\dots,u_\iota^\alpha,\dots,u_\ell] \]
 for each $\iota$ and each nontrivial join $\bigvee_\alpha u^\alpha_\iota$ in $\O X$. It is not hard to show that this indeed gives a presentation for $\O X^\N$ since the elements $\bigwedge_i [s(n_i) \in u_i]$ (with distinct $n_i$'s and each $u_i > 0$) form a base (see \cite[Proposition 2.3.7]{henry2016corrected}).
 
 The original presentation of the locale of left-total relations from $\B$ to $\N$ uses generators $[m(U) = k]$. We can easily obtain a presentation of the correct form by instead freely generating a $\wedge$-semilattice from this set of generators to obtain generators $[m(\vec{U}) = \vec{k}] = \bigwedge_j [m(U_j) = k_j]$ (with an order induced by the $\wedge$-semilattice structure). The ``$R_1$'' relations simply impose that the finite meets are respected, while for the ``$R_2$'' relations we have \[ [m(\vec{U}) = \vec{k}] \le \bigvee_{k_0 \in \N} [m(U_0,\vec{U}) = k_0,\vec{k}],\]
 so as to give left totality and force the necessary condition on $R_2$ relations to hold.
 
 Combining these we arrive at an appropriate presentation for $\ModCauchy(X)$. The generators are $[s(\vec{n}) \in \vec{u}] \wedge [m(\vec{U}) = \vec{k}]$ with the preorder induced by the product of the two preordered sets of generators from above. The ``$R_2$'' relations are then
 \begin{enumerate}
  \item $[s(\vec{n}) \in u_1,\dots,\bigvee\nolimits_\alpha u_\iota^\alpha,\dots] \wedge [m(\vec{U}) = \vec{k}] \le \bigvee_\alpha [s(\vec{n}) \in u_1,\dots,u_\iota^\alpha,\dots] \wedge [m(\vec{U}) = \vec{k}]$,
  \item $[s(\vec{n}) \in \vec{u}] \wedge [m(\vec{U}) = \vec{k}] \le \bigvee_{k_0 \in \N} [s(\vec{n}) \in \vec{u}] \wedge [m(U_0,\vec{U}) = k_0,\vec{k}]$,
  \item $[s(\vec{n}) \in \vec{u}] \wedge [m(U_0,\vec{U}) = k_0,\vec{k}] \le \bigvee_{\substack{w \in U_0 \\ \mathclap{w \between u_1,u_2}}} \, [s(\vec{n}) \in u_1 \wedge w, u_2 \wedge w, \vec{u}_{3\dots\ell}] \wedge [m(\vec{U}) = \vec{k}]$ \\
  for $n_1,n_2 \ge k_0$,
  \item[iii')] $[s(\vec{n}) \in \vec{u}] \wedge [m(U_0,\vec{U}) = k_0,\vec{k}] \le \bigvee_{\substack{w \in U_0 \\ \mathclap{w \between u_1}}} \, [s(\vec{n}) \in u_1 \wedge w, \vec{u}_{2\dots\ell}] \wedge [m(\vec{U}) = \vec{k}]$ \\
  for $n_1 \ge k_0$.
 \end{enumerate}
 For the condition (iii) we have restricted to the case where $n = n_1$ and $n' = n_2$ without loss of generality, since we may add in $u,u' = 1$ as necessary if $n,n' \notin \vec{n}$ or rearrange the indices to move $n$ and $n'$ to the front. This only fails if $n = n'$, in which case we instead use condition (iii').
 Also note that we can require the $w$'s in the join to satisfy $w \between u_1,u_2$, since the opens $[s(\vec{n}) \in \vec{v}]$ with each $v_i > 0$ form a base for $X^\N$.
 
 Again, it is not difficult to see that this indeed gives a presentation and that it is of the correct form to apply \cref{thm:sup_coverage}.
 So to show $q_\#$ to be well-defined we must prove it is monotone with respect the generators and that the relations (i), (ii), (iii) and (iii') are preserved.
 
 For monotonicity, first note that since $q_\#$ manifestly does not depend on the order of the lists, we may assume a generic inequality of generators to be of the form
 \[ [s(\vec{n}^1,\vec{n}^2) \in \vec{u}^1, \vec{u}^2] \wedge [m(\vec{U}^1,\vec{U}^2) = \vec{k}^1,\vec{k}^2] \le [s(\vec{n}^1,\vec{n}^4) \in \vec{u}^3, \vec{1}] \wedge [m(\vec{U}^1) = \vec{k}^1], \]
 where $u^1_i \le u^3_i$ for each $i$. To show $q_\#$ preserves this inequality, we consider a term in the join in the definition of $q_\#([s(\vec{n}^1,\vec{n}^2) \in \vec{u}^1, \vec{u}^2] \wedge [m(\vec{U}^1,\vec{U}^2) = \vec{k}^1,\vec{k}^2])$ from \cref{def:triquotiency}. Explicitly, we take $(V_{j^{1,2}})_{j^{1,2}}$ such that each $(V_{j^{1,2}})^\star \le U_{j^{1,2}}$ and $(v_{i^{1,2}})_{i^{1,2}}$ such that each $v_{i^{1,2}} \in \bigcap\{{\downarrow} V_{j^{1,2}} \mid k_{j^{1,2}} \le n_{i^{1,2}}\}$ and $v_{i^{1,2}} \between u_{i^{1,2}}$. Here we are using the convention that $i^1$ ranges over indices of $\vec{n}^1$, $i^2$ over $\vec{n}^2$, and $i^{1,2}$ over both (and similarly for $j^{1,2}$).
 We need to show that $\bigwedge_{i^{1,2}} [v_{i^{1,2}} \in F] \le q_\#([s(\vec{n}^1,\vec{n}^4) \in \vec{u}^3, \vec{1}] \wedge [m(\vec{U}^1) = \vec{k}^1])$. Let us consider the same $V_{j^1}$'s in the join for $q_\#$ on the right-hand side. We similarly take the same $v_{i^1}$'s. (Then $v_{i^1} \in \bigcap\{{\downarrow} V_{j^{1,2}} \mid k_{j^{1,2}} \le n_{i^1}\} \subseteq \bigcap\{{\downarrow} V_{j^1} \mid k_{j^1} \le n_{i^1}\}$, while $v_{i^1} \between u_{i^1}^1 \le u_{i^1}^3$ gives $v_{i^1} \between u_{i^1}^3$.)
 On the other hand, for the indices $i^4$ of $\vec{u}^4$ we take the join over \emph{all} $v_{i^4} \in \widetilde{V}_{i^4}$, where we define $\widetilde{V}_{i^4}$ to comprise the nonzero elements of $\bigcap\{{\downarrow} V_{j^1} \mid k_{j^1} \le n_{i^4}\}$. (We do, of course, have $v_{i^4} \between 1$, since $v_{i^4} > 0$.)
 Now we have
 \begin{align*}
  \bigwedge\nolimits_{i^{1,2}} [v_{i^{1,2}} \in F] &\le \bigwedge\nolimits_{i^1} [v_{i^1} \in F] \\
   &= \bigwedge\nolimits_{i^1} [v_{i^1} \in F] \,\wedge\, \bigwedge\nolimits_{i^4} \!\! \bigvee_{v_{i^4} \in \widetilde{V}_{i^4}}\!\! [v_{i^4} \in F] \\
   &= \!\!\! \bigvee_{(\vec{v}_{i^4})_{i^4} \in {\prod_{i^4} \widetilde{V}_{i^4}}} \!\!\! \bigwedge\nolimits_{i^{1,4}} [v_{i^{1,4}} \in F] \\
   &\le q_\#([s(\vec{n}^1,\vec{n}^4) \in \vec{u}^3, \vec{1}] \wedge [m(\vec{U}^1) = \vec{k}^1]),
 \end{align*}
 where the equality on the second line is by the Cauchyness axiom (iv) of \cref{prop:completion_via_filters}, which gives $\bigvee_{v_{i^4} \in \widetilde{V}_{i^4}} [v_{i^4} \in F] = 1$ (as $\widetilde{V}_{i^4}$ is a uniform cover).
 Thus, we have shown monotonicity.
 
 Next we will prove that $q_\#$ (as defined in \cref{def:triquotiency}) respects the relations.
 
 i) As above, we take $V_1, \dots, V_{\abs{\vec{k}}}$ such that each $V_j^\star \le U_j$ and $v_i$'s with $v_i \in \bigcap\{ {\downarrow} V_j \mid k_j \le n_i\}$ such that $v_i \between u_i$ for $i \ne \iota$ and $v_\iota \between \bigvee_\alpha u_\iota^\alpha$. Then $v_\iota \between u_\iota^\alpha$ for some $\alpha$. In the join on the right-hand side we consider the same $V_j$'s and the same $v_i$'s. It is then clear that $\bigwedge_i [v_i \in F]$ occurs in the join obtained by expanding $\bigvee_\alpha q_\#([s(\vec{n}) \in u_1,\dots,u_\iota^\alpha,\dots,u_\ell] \wedge [m(\vec{U}) = \vec{k}])$, as required.
 
 ii) We proceed in the same way as before, taking the same $v_i$'s and $V_j$'s on the right-hand side as on the left-hand side for $j > 0$. We then choose an arbitrary uniform cover $V_0$ with $V_0^\star \le U_0$. As long as we take $k_0$ larger than every $n_i$, we still have $v_i \in \bigcap\{ {\downarrow} V_j \mid k_j \le n_i\}$ on the right-hand-side, since the $j = 0$ case never actually occurs.
 
 iii) As we are used to by now, consider the $V_j$'s and $v_i$'s for the left-hand side and make the same choices for the right-hand side.
 It remains to choose a $w \in U_0$ that ensures we have $v_1 \between w \wedge u_1$ and $v_2 \between w \wedge u_2$ (and hence also $w \between u_1,u_2$).
 Suppose for now that $v_1 \between v_2$. Recall that $n_1,n_2 \ge k_0$ by assumption and so $v_1, v_2 \in {\downarrow} V_0$. Then $v_1 \vee v_2 \le \st(v_1, V_0) \in {\downarrow} V_0^\star \le U_0$. Thus, there is a $w \in U_0$ such that $w \ge v_1 \vee v_2$. Then for this $w$ and for $i \in \{1,2\}$, we have $v_i \wedge w \wedge u_i \ge v_i \wedge (v_1 \vee v_2) \wedge u_i = v_i \wedge u_i > 0$, since $v_i$ was chosen such that $v_i \between u_i$.
 
 The equality we are trying to prove is $\bigwedge_i [v_i \in F] \le \bigvee_{\substack{w \in U_0 \\ \mathclap{w \between u_1,u_2}}} \, q\#([s(\vec{n}) \in u_1 \wedge w, u_2 \wedge w, \vec{u}_{3\dots\ell}])$.
 On the left-hand side we have $\bigwedge_i [v_i \in F] = [\bigwedge_i v_i \in F] = \bigvee\{[\bigwedge_i v_i \in F] \mid \bigwedge_i v_i > 0\}$ by the relations (ii) and (iii) of \cref{prop:completion_via_filters}. So to prove the inequality, we may assume $\bigwedge_i v_i > 0$, and hence in particular, $v_1 \between v_2$. This justifies the assumption above and so we have proved the claim.
 
 iii') This is just like case (iii) with $w \in U_0$ chosen so that $w \ge v_1$. The result follows.
\end{proof}

\begin{theorem}
 Let $X$ be a pre-uniform locale with base $\B$. The limit map $q\colon \ModCauchy(X) \to \Cvar X$ is a lower triquotient map, with triquotiency assignment $q_\#$.
\end{theorem}
\begin{proof}
 Certainly, $q_\#(1) = 1$. We may check the Frobenius condition $q_\#(a \wedge q^*(b)) = q_\#(a) \wedge b$ on a base. It is equivalent to show $q_\#(a \wedge q^*(b)) \ge q_\#(a) \wedge b$ and $q_\#q^*(b) \le b$ and we will do this for $a = [s(\vec{n}) \in \vec{u}] \wedge [m(\vec{U}) = \vec{k}]$ and $b = [u \in F]$.
 
 First note that
 \begin{align*}
  q_\#(q^*([u \in F])) &= \bigvee_{U \in \B \vphantom{k'}} \bigvee_{\;\substack{v \vartriangleleft_U u' \vartriangleleft u \vphantom{k'} \\ v > 0}} \bigvee_{\; k' \le k \in \N} q_\#([s(k) \in v] \wedge [m(U) = k']) \\
    &= \bigvee_{U \in \B \vphantom{k'}} \bigvee_{\;\substack{v \vartriangleleft_U u' \vartriangleleft u \vphantom{k'} \\ v > 0}} \bigvee_{\; k' \le k \in \N} \bigvee_{\substack{V \in \B \\ V^\star \le U}}\, \bigvee_{\substack{v' \in V \\ v' \between v}} [v' \in F].
 \end{align*}
 (Here we have assumed $v > 0$ in the join by writing $v$ as a join of such nonzero elements and using that $v \mapsto [s(k) \in v]$ preserves joins.)
 So to show $q_\#q^*([u \in F]) \le [u \in F]$ we take $U \in \B$, $v \vartriangleleft_U u' \vartriangleleft u$ with $v > 0$, $k' \le k \in \N$, $V \in \B$ with $V^\star \le U$ and $v' \in V$ with $v' \between v$, and show that $[v' \in F] \le [u \in F]$.
 Since $v' \between v$ and $v' \in V$, we have $v' \le \st(v, V) \le \st(v, U) \le u' \le u$. So by axiom (ii) of \cref{prop:completion_via_filters}, we find that $[v' \in F] \le [u \in F]$, as required.
 
 Next we show that $q_\#([s(\vec{n}) \in \vec{u}] \wedge [m(\vec{U}) = \vec{k}]) \wedge [u \in F] \le q_\#([s(\vec{n}) \in \vec{u}] \wedge {[m(\vec{U}) = \vec{k}]} \wedge q^*([u \in F]))$.
 Consider basic uniform covers $V_1,\dots,V_{\abs{\vec{k}}}$ with $V_j^\star \le U_j$ and a family of opens $(v_i)_i$ such that $v_i \in \bigcap\{ {\downarrow} V_j \mid k_j \le n_i\}$ and $v_i \between u_i$. We want that $[u \in F] \wedge \bigwedge_i [v_i \in F]$ is less than or equal to the right-hand side of the desired inequality.
 
 Expanding the right-hand side we have
 \begin{align*}
  \text{RHS}\, = \bigvee_{U \in \B \vphantom{k'}} \bigvee_{\; \substack{v \vartriangleleft_U u' \vartriangleleft u \vphantom{k'} \\ v > 0}} \bigvee_{\; k' \le k \in \N} q_\#( [s(k) \in v] \wedge [s(\vec{n}) \in \vec{u}] \wedge {[m(U,\vec{U}) = k',\vec{k}]} ).
 \end{align*}
 Take $k'$ strictly larger than each $n_i \in \vec{n}$, so that in particular, $k \notin \vec{n}$.
 Then the term $q_\#( [s(k) \in v] \wedge {[s(\vec{n}) \in \vec{u}]} \wedge {[m(U,\vec{U}) = k',\vec{k}]} )$ reduces to
 \begin{align*}
  q_\#( [s(k,\vec{n}) \in v,\vec{u}] \wedge {[m(U,\vec{U}) = k',\vec{k}]} ) &= \!\bigvee_{\substack{\vec{V} \in \B^{\abs{\vec{k}}+1} \\ V_j^\star \le U_j}}\, \bigvee_{\substack{\vec{v} \in (\cup \B)^{\abs{\vec{n}}+1} \\ v_i \in {\downarrow} V_j \text{ for } k_j \le n_i \\ v_i \between u_i}} \!\! \bigwedge\nolimits_{i \ge 0} [v_i \in F] \\
  &= \!\bigvee_{\substack{\vec{V} \in \B^{\abs{\vec{k}}+1} \\ V_j^\star \le U_j}}\, \bigvee_{\substack{\vec{v} \in (\cup \B)^{\abs{\vec{n}}+1} \\ v_i \in {\downarrow} V_j \text{ for } k_j \le n_i \\ v_i \between u_i}} \!\! [v_0 \in F] \wedge \bigwedge\nolimits_{i > 0} [v_i \in F],
 \end{align*}
 where we define $n_0 = k$, $u_0 = v$, $U_0 = U$ and $k_0 = k'$.
 Take the $V_j$'s and $v_i$'s in the join to be the same as above when $j > 0 $ and $i > 0$.
 Then we see
 \begin{align*}
  \text{RHS}\, &\ge \left( \bigvee_{U \in \B \vphantom{k'}} \bigvee_{\; \substack{v \vartriangleleft_U u' \vartriangleleft u \vphantom{k'} \\ v > 0}}\, \bigvee_{\substack{V_0 \in \B \vphantom{k'} \\ V_0^\star \le U}}\, \bigvee_{\substack{v_0 \in V_0 \vphantom{k'} \\ v_0 \between v}} [v_0 \in F] \right) \wedge \bigwedge_{i > 0} [v_i \in F].
 \end{align*}
 It now suffices to show that the bracketed expression lies above $[u \in F]$ (since we explicitly take the meet with $\bigwedge_{i > 0} [v_i \in F]$).
 First note that 
 \[\bigvee_{\substack{v_0 \in V_0 \\ v_0 \between v}} [v_0 \in F] \ge \!\!\bigvee_{\substack{v_0 \in V_0 \\ v_0 \between v}} [v \wedge v_0 \in F] = \!\!\bigvee_{v_0 \in V_0} [v \wedge v_0 \in F] = [v \in F] \wedge \!\! \bigvee_{v_0 \in V_0} [v_0 \in F] = [v \in F],\]
 by using (iii), (ii) and (iv) of \cref{prop:completion_via_filters} in turn.
 Then we find
 \begin{align*}
\bigvee_{U \in \B \vphantom{k'}} \bigvee_{\; \substack{v \vartriangleleft_U u' \vartriangleleft u \vphantom{k'} \\ v > 0}} \, \bigvee_{\substack{V_0 \in \B \vphantom{k'} \\ V_0^\star \le U}}\, \bigvee_{\substack{v_0 \in V_0 \vphantom{k'} \\ v_0 \between v}} [v_0 \in F]
  \ge \! \bigvee_{U \in \B \vphantom{k'}} \bigvee_{\; \substack{v \vartriangleleft_U u' \vartriangleleft u \vphantom{k'} \\ v > 0}} \!\! [v \in F]
  = \bigvee_{\substack{\mathclap{v \vartriangleleft u' \vartriangleleft u} \\ v > 0}}\, [v \in F]
  = [u \in F],
\end{align*}
by the regularity axiom (v) together with (iii).
 The result follows.
\end{proof}

Thus, we have obtained the completion of the (pre-)uniform locale $X$ as a quotient of the locale of modulated Cauchy sequences on $X$.
Note that the completion map $\gamma_X\colon X \to \Cvar X$ can also be expressed in these terms.
\begin{proposition}
 The completion map $\gamma_X\colon X \to \Cvar X$ is obtained as a composite of the constant sequence map $c: X \to \ModCauchy(X)$, defined by
 \[c^*([s(n) \in u] \wedge [m(U) = k]) = u,\]
 and the limit map $q\colon \ModCauchy(X) \to \Cvar X$.
\end{proposition}
\begin{proof}
 First note that $c^*$ is indeed a well-defined frame homomorphism.
 We then have
 \begin{align*}
  c^*q^*([u \in F]) \,= \bigvee_{U \in \B \vphantom{k'}} \bigvee_{\; v \vartriangleleft_U u' \vartriangleleft u \vphantom{k'}} \bigvee_{\; k' \le k \in \N} v \,= \bigvee_{v \vartriangleleft u} v \,=\, \gamma^*([u \in F]),
 \end{align*}
 as required.
\end{proof}
We can also define a pre-uniform structure on $\ModCauchy(X)$ corresponding to the uniformity on $\Cvar X$.\footnote{Constructively, pre-uniform locales behave best when the underlying locale is overt. We can prove that $\ModCauchy(X)$ is overt by constructing an explicit left adjoint $\exists$ to ${!}^*\colon \O 1 \to \O\ModCauchy(X)$. We define $\exists([s(\vec{n}) \in \vec{u}] \wedge [m(\vec{U}) = \vec{k}]) = \llbracket \exists (w_{ii'}^j)_{jii'}.\, (\forall j.\, w_{ii'}^j \in U_j) \land (\forall i.\, (\bigwedge_{\!\!\!\substack{i',j \quad \\ k_j \le n_i,n_{i'}}} \hspace{-3ex} w_{i'i}^j \wedge w_{ii'}^j) \between u_i) \land \bigwedge_{\!\! \substack{i,j \ \\ k_j \le n_i}} \hspace{-1ex} w_{ii}^j > 0 \rrbracket$.}
\begin{definition}
 A natural uniformity on $\ModCauchy(X)$ is generated by covers of the form $\{q^*([u \in F]) \mid u \in U\} = \{ \bigvee_{U' \in \B \vphantom{k'}} \bigvee_{\; v \vartriangleleft_{U'} u' \vartriangleleft u \vphantom{k'}} \bigvee_{\; k' \le k \in \N} [m(U') = k'] \wedge [s(k) \in v] \mid u \in U \}$ for each basic uniform cover $U$ of $X$. This is the initial uniformity with respect to $q$.
\end{definition}

At this point we no longer need the original construction of the completion in terms of regular Cauchy filters. In particular, by the following lemma, $\Cvar X$ is the uniform reflection of $\ModCauchy(X)$.
\begin{lemma}
 Suppose $(X,\U)$ is a uniform locale and $e\colon Y \to X$ is a locale epimorphism. Equip $Y$ with the initial uniformity with respect to the forgetful functor $\PreUnifLoc \to \Loc$. Then $e\colon Y \to X$ is (up to isomorphism) the unit of the uniform reflection of $Y$.
\end{lemma}
\begin{proof}
 The uniform reflection of $Y$ is the subframe of $\O Y$ consisting of the elements $a \in \O Y$ such that $a = \bigvee_{b \vartriangleleft a} b$.
 Since $X$ is uniform, for every $u \in \O X$ we have $u = \bigvee_{v \vartriangleleft u} v$. Since $e^*$ is uniform, it preserves $\vartriangleleft$ and so $e^*(u) = \bigvee_{v \vartriangleleft u} e^*(v) \le  \bigvee_{e^*(v) \vartriangleleft e^*(u)} e^*(v) \le e^*(u)$. Thus, $e^*(\O X)$ is contained in the subframe corresponding to the uniform reflection of $Y$. 
 
 Conversely, consider an element $a$ of the uniform reflection of $Y$. Recall that ${b \vartriangleleft a}$ means $b \le \st(b,e^*[U]) \le a$ for some $U \in \U$. Hence, $a = \bigvee_{U \in \U}\bigvee_{b \vartriangleleft_{e^*[U]} a} \st(b,e^*[U])$. But $\st(b,e^*[U]) = \bigvee\{ e^*(w) \mid w \in U,\ b \between e^*(w) \}$ and is thus an element of $e^*(\O X)$. Consequently, $a$ is also an element of $e^*(\O X)$.
 Now, since $e^*$ is injective, the result follows.
\end{proof}

Alternatively, we can define $\Cvar X$ as the coequaliser of the kernel pair of $q^*$. This kernel pair can be specified without reference to $\Cvar X$ as the projections out of the sublocale of $\ModCauchy(X) \times \ModCauchy(X)$ cut out by the relations $\bigvee_{U \in \B \vphantom{k'}} \bigvee_{\; v \vartriangleleft_U u' \vartriangleleft u \vphantom{k'}} \bigvee_{\; k' \le k \in \N} \iota_1({[m(U) = k']} \wedge [s(k) \in v]) = \bigvee_{U \in \B \vphantom{k'}} \bigvee_{\; v \vartriangleleft_U u' \vartriangleleft u \vphantom{k'}} \bigvee_{\; k' \le k \in \N} \iota_2([m(U) = k'] \wedge [s(k) \in v])$ for $u \in \O X$, where $\iota_1$ and $\iota_2$ are the frame coproduct injections.

Finally, if we desire a new presentation of $\Cvar X$, the techniques of \cite{manuell2023presenting} can be used to obtain a presentation of $\Cvar X$ from that of $\ModCauchy(X)$ and the composite map $q^*q_\#\colon \O\ModCauchy(X) \to \O\ModCauchy(X)$, which sends $[s(\vec{n}) \in \vec{u}] \wedge [m(\vec{U}) = \vec{k}]$ to 
\[\bigvee_{\substack{\vec{V} \in \B^{\abs{\vec{k}}} \\ V_j^\star \le U_j}}\, \bigvee_{\substack{\vec{v} \in (\cup \B)^{\abs{\vec{n}}} \\ v_i \in {\downarrow} V_j \text{ for } k_j \le n_i \\ v_i \between u_i}} \bigvee_{U \in \B \vphantom{\bigwedge}} \bigvee_{\; w \vartriangleleft_U v' \vartriangleleft \,\bigwedge\nolimits_i \! v_i} \bigvee_{\; k' \le k \in \N \vphantom{\bigwedge}} [m(U) = k'] \wedge [s(k) \in w].\]
However, the resulting presentation is probably too complicated to be of much use.

\section{Discussion}

\subsection{Relation to the spatial completion}

It might seem puzzling how Cauchy sequences can work for pre-uniform locales, but apparently not for uniform spaces, which are a special case. Of course, in our approach the intermediate constructions involve locales instead of spaces, but it is helpful to consider the correspondence in more detail.

Firstly, note that while the completion of a spatial pre-uniform locale can fail to be spatial, its spatial coreflection always agrees with the classical completion of the corresponding uniform space. So our construction indeed recovers the correct uniform completion of a uniform space. On the other hand, the locale $\ModCauchy(X)$ will generally fail to be spatial, even if $X$ is. Its spatial coreflection indeed gives a space of (modulated) Cauchy sequences, and the corresponding quotient of this space is the sequential completion, not the full uniform completion. The resolution to this discrepancy is that the spatial coreflection \emph{does not preserve coequalisers}.
One might conclude that the failure of sequences to give the correct completion of uniform spaces is a pathology arising from the `taking points too early'.

Our approach also sheds light on some cases where the sequential completion \emph{does} work in the spatial setting. If $X$ is a Polish space and $\B$ is a countable base for a uniformity on $X$, then (assuming classical logic) $\ModCauchy(X)$ is spatial by \cite{heckmann2015}.
Then the spatial coequaliser agrees with the localic version and the sequential completion and full uniform completion coincide.

\subsection{Applications}

The motivation behind this result is to reconcile the sequential and filter approaches to completions. That these two approaches give the same result has a certain aesthetic charm and gives us license to not worry about there being a variant notion of completion in the pointfree setting.

Unfortunately, the construction via sequences is significantly more complicated than the standard construction using filters and so can hardly be preferred to construct completions in practice.
Nonetheless, it does have some potential uses by making contact with classical constructions that use sequences.
For example, if function into $\R$ is defined by power series or using some other kind of convergent sequence, we can define it in the pointfree setting by specifying a map into $\ModCauchy(\R)$ and composing with the limit map.

There is also some promise in being able to use similar definitions even in non-metric settings, where in point-set topology sequences are no longer of much value.

\subsection{Further prospects for Skolemisation}

The approach we used in \cref{sec:locale_of_cauchy_sequences} to define the locale of modulated Cauchy sequences seems like it could be very useful as a general technique to describe classifying locales for objects defined by logically complex formulae.
This is especially true when we proceed to quotient out the additional data, as we did here, but I imagine it could be helpful even if we are required to keep the additional data around.
A full discussion of this technique will be the topic of a later paper.

\bibliographystyle{abbrv}
\bibliography{references}

\end{document}